\documentclass{amsart}
\usepackage{graphicx,amssymb,amsmath,hyperref}
\newtheorem{thm}{Theorem}[section]

\newtheorem{lem}[thm]{Lemma}

\newtheorem{qus}[thm]{Question}
\theoremstyle{definition}

\theoremstyle{remark}

\numberwithin{equation}{section}

\begin{document}
\title[]{Non-solvable groups generated by involutions in which every involution is left $2$-Engel}
\author{Alireza Abdollahi}%
\address{Department of Mathematics, University of Isfahan, Isfahan 81746-73441, Iran; and School of Mathematics, Institute for Research in Fundamental Sciences (IPM), P.O. Box 19395-5746, Tehran, Iran}%
\email{a.abdollahi@math.ui.ac.ir}%

\thanks{}%
\subjclass[2000]{20D15;20F12}%
\keywords{Groups of exponent 4; Involution; Left $2$-Engel; Abelian normal closure }%

\begin{abstract}
The following problem is proposed as Problem 18.57 in [The Kourovka Notebook, No. 18, 2014] by D. V. Lytkina:\\
Let $G$ be a finite $2$-group generated by involutions in which $[x, u, u] = 1$ for
every $x \in G$ and every involution $u \in G$. Is the derived length of $G$ bounded?\\
The question is asked of an upper bound on the solvability length of finite $2$-groups
generated by involutions in which every involution (not only the generators) is also left $2$-Engel. We negatively answer the question.
\end{abstract}
\maketitle
\section{\bf Introduction and Result}
The following problem is proposed as Problem 18.57 of \cite{Kour} by D. V. Lytkina:
\begin{qus}\label{qus} Let $G$ be a finite $2$-group generated by involutions in which $[x, u, u] = 1$ for
every $x \in G$ and every involution $u \in G$. Is the derived length of $G$ bounded?
\end{qus}
Question \ref{qus} is asked of an upper bound on the solvability length of finite $2$-groups
generated by involutions in which all involutions of groups (not only the generators) are also left $2$-Engel elements. We negatively answer the question. In the proof we need some well-known facts about the groups of exponent $4$.\\
It is know that groups of exponent $4$ are locally finite \cite{S} and the free Burnside group $\mathfrak{B}$ of exponent $4$ with infinite countable rank is not solvable \cite{R}. In \cite{GW}, it is proved that the solvability of $\mathfrak{B}$ is equivalent to the one of the group $H$ defined as follows:\\
Let $H$ be the freest group generated by elements $\{x_i \;|\; i\in \mathbb{N}\}$ with respect to the following relations:\\
(1) \; $x_i^2=1$ for all $i\in\mathbb{N}$; \\
(2) \; The normal closure $\langle x_i \rangle^H$ is abelian for all $i\in \mathbb{N}$.\\
(3) \; $h^4=1$ for all $h\in H$; \\
Therefore $H$ is a non-solvable group of exponent $4$ generated by involutions $x_i$ ($i\in \mathbb{N})$. Note that the relation (2) above is equivalent to  say that $x_i$ is a left 2-Engel element of $H$ that is  $[x,x_i,x_i]=1$ for all $x\in H$. We do not know if $H$ has the property requested in Question \ref{qus}, that is, whether every involution $u\in H$ is a left $2$-Engel element of $H$. Instead we find a quotient of $H$ which is still non-solvable but it satisfies the latter property. The latter quotient of $H$ will provide a counterexample for Question \ref{qus}. To introduce the quotient we need to recall some definitions and results on right $2$-Engel elements. \\
For any group $G$, $R_2(G)$ denotes the set of all right $2$-Engel elements of $G$, i.e.
$$R_2(G)=\{a\in G \;|\; [a,x,x]=1 \;\text{for all}\; x\in G\}.$$
It is known  \cite{K} that $R_2(G)$ is a characteristic subgroup of $G$. The subgroup $R_2(G)$ is a $2$-Engel group that is $[x,y,y]=1$ for all $x,y\in R_2(G)$. Thus $R_2(G)$ is nilpotent of class at most $3$ \cite{H} and so it is  of solvable length at most $2$.
\begin{thm}\label{thm}
Let $H$ be the freest group defined above.  Then $\overline{H}=H/R_2(H)$ satisfies the following condition:\\
{\rm (4)} \; all involutions $u\in\overline{H}$ are left $2$-Engel in $\overline{H}$.\\
 Furthermore, $\overline{H}$ is not solvable so that there is no upper bound on  solvability lengths of finite $2$-groups  $\overline{H}_n=\frac{\langle x_1,\dots, x_n\rangle}{R_2(\langle x_1,\dots, x_n\rangle)}$ which satisfy all conditions {\rm (1), (2), (3)} and {\rm (4)} above.
\end{thm}

\section{\bf Proof of Theorem \ref{thm}}
The following is the key lemma of the paper.
\begin{lem}\label{lem}
Let $G$ be any group of exponent $4$ and $b\in G$ is such that $b^2\in R_2(G)$. Then $[a,b,b]\in R_2(G)$ for all $a\in G$.
This means that, in every group $G$ of exponent $4$,  every involution of the quotient $G/R_2(G)$ has an abelian normal closure.
\end{lem}
\begin{proof}
Let $N$ be the freest group generated by elements $a,b,c$ subject to the following relations:\\
(1) \; $x^4=1$ for all $x\in N$; \\
(2) \; $[b^2,x,x]=1$ for all $x\in N$.\\
By \cite{S} it is known that $N$ is finite. Now by  {\sf nq} package \cite{N} one can construct $N$ in {\sf GAP} \cite{GAP} by the following commands:
\begin{verbatim}
LoadPackage("nq");
F:=FreeGroup(4);
a:=F.1;b:=F.2;c:=F.3;x:=F.4;
G:=F/[x^4,LeftNormedComm([b^2,x,x])];
N:=NilpotentQuotient(G,[x]);
gen:=GeneratorsOfGroup(G,[x]);
LeftNormedComm([gen[1],gen[2],gen[2],gen[3],gen[3]]);
\end{verbatim}
Note that in above {\tt gen[1]}, {\tt gen[2]} and {\tt gen[3]} correspond to the free generators $a,b$ and $c$, respectively.
The output of last command in above (which is {\tt id} the trivial element of $N$) shows that $[a,b,b,c,c]=1$. This completes the proof.
\end{proof}
 It may be interesting in its own right that   the group $N$ defined in the proof of Lemma \ref{lem} is nilpotent of class $7$ and order $2^{41}$.\\

\noindent{\bf Proof of Theorem \ref{thm}.} It follows from Lemma  \ref{lem} that $\overline{H}=H/R_2(H)$ has the property (4) mentioned in the statement of Theorem \ref{thm}.
Since  $H$ is not solvable by \cite{GW} and \cite{R}, it follows that  there is no upper bound on the solvable lengths of finite $2$-groups $\overline{H}_n$.
By construction $\overline{H}_n$ is generated by involutions and by Lemma \ref{lem} all involutions in $\overline{H}_n$ are left $2$-Engel. This completes the proof. $\hfill\Box$ \\

\section*{\bf Acknowledgements}
This research was in part supported by a grant  (No. 92050219) from School of Mathematics, Institute for Research in Fundamental Sciences (IPM). The author gratefully acknowledges the financial support of the Center of Excellence for Mathematics, University of Isfahan.

\end{document}